\documentclass[11pt]{tran-l}
\usepackage{amsmath}
\usepackage{amsfonts,amssymb,amsmath}

\newtheorem{theorem}{Theorem}[section]
\newtheorem{lemma}[theorem]{Lemma}
\newtheorem{prop}[theorem]{Proposition}
\newtheorem{corollary}[theorem]{Corollary}

\theoremstyle{definition}
\newtheorem{definition}[theorem]{Definition}

\theoremstyle{remark}

\numberwithin{equation}{section}

\begin{document}

\title[Entropies of Smooth Flows]{Topological Entropies of Equivalent Smooth Flows}

\author{Wenxiang Sun}
\address{LMAM, School of Mathematical Sciences,\ Peking University,  China}
\email{sunwx@math.pku.edu.cn}

\author{Todd Young}
\address{Department of Mathematics, Ohio University, U.S.A.}
\email{young@math.ohiou.edu}

\author{Yunhua Zhou}
\address{School of Mathematical Sciences,\ Peking University,  China}
\email{zhouyh@math.pku.edu.cn}
 \thanks{The first author was supported by NSFC (\#10231020, \#10671006) 
 and National Basic Research Program of China (973 Program) (\# 2006CB805900).
 The second author was supported by an Ohio University Faculty Fellowship Leave. 
 The third author was supported by NSFC (\#10671006).}


\begin{abstract}
We construct two equivalent smooth flows,
one of which has positive topological entropy and the other has zero
topological entropy.
This provides a negative answer to a problem posed by Ohno.  \\
\end{abstract}

\maketitle

\section{Introduction}

Two flows defined on a smooth manifold are {\em equivalent} if there exists a
homeomorphism of the manifold that sends each
orbit of one flow  onto an orbit of the other flow while preserving
the time orientation.
The {\em topological entropy} of a flow is
defined as the entropy of its time-1 map.
While topological entropy is an invariant
for equivalent homeomorphisms (Theorem 7.2 in
\cite{w}),  finite non-zero topological entropy for a flow
cannot be an invariant because its value is affected by
time reparameterization. However, 0  and
$\infty$ topological entropy are invariants for
equivalent flows {\em without} fixed points (see
\cite{ohno}\cite{sv}\cite{tho1}\cite{tho2}).

In equivalent flows {\em  with} fixed points
there exists a counterexample, constructed by Ohno
\cite{ohno}, showing that neither 0 nor $\infty$ topological
entropy is preserved by equivalence. The two flows constructed in \cite{ohno} are
suspensions of a transitive subshift and thus are not differentiable.
Note that  a differentiable flow on  a compact manifold cannot have
$\infty$ entropy (see Theorem 7.15 in \cite{w}).
These facts led Ohno \cite{ohno} in 1980 to ask the following:
\begin{quotation}
\noindent
{\em Is $0$ topological entropy an invariant for
equivalent differentiable flows?}
\end{quotation}
In this paper, we construct two equivalent 
$C^\infty$ smooth flows with a singularity, one of which has positive
topological entropy while the other has zero topological entropy. This
gives a negative answer to Ohno's question.

We begin by supposing that $f: M\to M$  is a
$C^\infty$ diffeomorphism of a smooth compact Riemannian
manifold $M$ with $\dim M=m \geq 2$ with the following properties:
(1) $f$ has positive topological entropy and (2) $f$ is minimal in
the sense that all forward orbits are dense in $M$ (or equivalently
closed invariant sets are either empty or the entire space). An example of
such an $f$ was constructed by Herman \cite{her}.
Using the constant function
$I: M\to \mathbb{R},\,\, I(x)\equiv 1$, one gets a suspension manifold
$\Omega$ and a smooth suspension flow $\psi: \Omega\times
\mathbb{R} \to \Omega$ (Definition~\ref{definition:1}). Let $X$ denote
the smooth vector field associated with this flow. For any function
$\alpha\in C^\infty(\Omega, [0,1])$, note that $\alpha X$ is  a $C^\infty$
vector field on $\Omega$ and it thus induces a unique, differentiable flow.

\begin{theorem}{\em (Main Theorem)} \label{theorem:1}
There exist two functions $\alpha, \widehat{\alpha} \in
C^\infty(\Omega,[0,1])$ satisfying the following:
\begin{enumerate}
\item $\alpha X$ and $\widehat{\alpha} X$  induce equivalent flows
$\varphi(t)$ and $\widehat{\varphi}(t)$;
\item  $\varphi(t)$ has zero topological entropy and
$\widehat{\varphi}(t)$ has positive topological entropy.
\end{enumerate}
\end {theorem}

The key idea used by Ohno and in this work is that the time
reparameterization between an orbit in one flow and its image
orbit in an equivalent flow can grow super-exponentially near a fixed
point. In the present case the main challenge is to ensure
that the time parameterizations lead to differentiable flows.
Throughout the rest of the paper we fix one point
$p=(x_0, 0)\in \Omega$ and consider functions
$\alpha, \widehat{\alpha} \in C^\infty(\Omega, [0,1])$ satisfying the following criteria:

\medskip
\noindent
{\bf (H)} We assume that $\alpha(p) = \widehat{\alpha}(p) =0$,
$\alpha(q)$ and $\widehat{\alpha}(q) > 0$,  for
$q \neq p \in \Omega$, and there exists a  small neighborhood $V$ of $p$ in
$\Omega$ such that $\alpha(q) = \widehat{\alpha}(q) \equiv 1$ if
$q\in \Omega\setminus V$.

In our construction $\widehat{\varphi}$ will be shown to have an ergodic probability
measure that is equivalent to an ergodic probability preserved by the suspension
flow of $f$. Positive entropy of $f$ will then imply positive entropy for
$\widehat{\varphi}$. On the other hand, in the construction of $\varphi$,
$\alpha$ will be very flat so that orbits near $p$ will be strongly slowed
down. This together with minimality, which ensures fast returns to any neighborhood
of $p$, will imply that $\phi$ has Dirac measure at $p$ as its unique
ergodic probability measure.


\section{Preliminaries}

\subsection{Notation}

The symbol $<\cdot, \cdot>$ will be used to denote the
Riemannian metric on either $M$ or $\Omega$ depending on the
context. We will denote a ball in $M$ by
$B_M(x,\epsilon)=\{y\in M: d(x, y)< \epsilon \}$
and by $B_\Omega(x,\epsilon)$ a ball in $\Omega$.
A ball in $\mathbb{R}^n$ centered at the origin
we will denote by  $B^n(r) = \{x\in \mathbb{R}^n: |x| <r \}$.

\begin {definition}\label{definition:1}
Suppose that $M$ is a smooth compact Riemannian manifold and that $f$ is
a $C^\infty$ diffeomorphism. Consider the space
$$
\Omega= M \times [0,1] / \sim,
$$
where $\sim$ is the identification of $(y,1)$ with $(f(y),0)$.
The {\em standard suspension} of $f$ is the flow
$\psi_t$ on $\Omega$ defined  by $\psi_t(y, s)=(y, t+s)$, for $0 \leq t+s<1.$
\end {definition}
A standard argument as in \cite{liao} shows that $\Omega$ is a
smooth compact Riemannian manifold and $\psi_t$ is $C^\infty$.
If  $f: M\to M$ is minimal as a homeomorphism, then  $\psi_t$
is a minimal flow.

\begin{prop}\label{remark:2}
If $\alpha$ and $\widehat{\alpha}$ in $C^\infty(\Omega, [0,1])$ satisfy  {\bf (H)}
and $X$ is the suspension vector field on $\Omega$ described above, then
$\alpha X$ and $\widehat{\alpha} X$ induce equivalent differential flows on
$\Omega$ with one singularity $p$.
\end{prop}

\begin{proof}
That a $C^\infty$ vector field induces a $C^\infty$ flow is a standard result.
The identity map on $\Omega$ takes orbits of one flow to orbits
of the other since the unique singular point $p$ is mapped to itself and
elsewhere $\alpha$ and $\widehat{\alpha}$ are positive.
The assumption that $\alpha$ and $\widehat{\alpha}$ are
non-negative also implies preservation of time orientation.
\end{proof}

\subsection{Minimal homeomorphisms and uniform recurrence}

The following result concerning uniform recurrence of orbits in minimal
homeomorphisms will be used in the proof that $\varphi$ has zero
entropy. It guarantees fast return of orbits to a neighborhood of
the stopped point $p$.

\begin{lemma}\label{lemma:1}
Let  $M$ denote  a smooth compact Riemannian manifold and let 
$f: M\rightarrow M$ be  a $C^\infty$ diffeomorphism. Suppose $(M,f)$ is
a minimal homeomorphism, then for any $\varepsilon>0$, there exists
$L(\varepsilon)\in \mathbb{N}$ such that, for any $x,y\in M$,  
$f^l(y)\in B(x,\varepsilon)$, for some $l$, $0 \leq l \leq L(\varepsilon)$.
\end {lemma}

\begin{proof}
This is a classical result
(see \cite{fur}). We present a
proof for the convenience of the readers.

For any $x, y\in M$ and $\epsilon >0$, 
there exists $l(x,y,\varepsilon) \in \mathbb{N}$
such that $f^{l(x,y,\varepsilon)}(y)\in B(x,\frac{\varepsilon}{2})$,
since $f$ is minimal. It follow that there exists $\delta(y)>0$ such that
$f^{l(x,y,\varepsilon)}(B(y, \delta(y)))
\subset B(x,\frac{3\varepsilon}{4})
\subset B(x,\varepsilon)$. By the
compactness of $M$, $\exists\ y_1, \cdots, y_n$ and  
$L(x,\varepsilon)=\max\limits_{i}\{l(x,y_i,\varepsilon)\}$ such that
\begin{enumerate}
\item[(i)] $\cup_{i=1}^{n}B(y_i, \delta(y_i))=M$;
\item[(ii)] $\forall y\in M$, $\exists \,\,l(y)\leq L(x, \varepsilon)$
such that $f^{l(y)}(y)\in B(x,\frac{3\varepsilon}{4})\subset
B(x,\varepsilon)$.
\end{enumerate}
Then, for any $\widetilde{x}\in B(x, \frac{\varepsilon}{4})$,
$f^{l(y)}(y)\in B(\widetilde{x},\varepsilon)$. The compactness of
$M$ implies that there exist finite points $x_1, \cdots, x_{n_0}$
such that $\cup_{i=1}^{n_0}B(x_i, \frac{\varepsilon}{4})=M$. Then,
$L(\varepsilon)=\max\limits_{i}\{L(x_i, \varepsilon)\}$ satisfies
the Lemma.
\end{proof}

Uniform recurrence as in Lemma~\ref{lemma:1} allows us to put
lower bounds on measures of sets with respect to ergodic measures.
\begin{lemma}\label{lemma:2}
Suppose the assumptions of Lemma~\ref{lemma:1} and
assume that $\mu$ is an ergodic measure of $f$, then
$$
\mu(B_M(x, \varepsilon))\geq \frac{1}{L(\varepsilon)}>0, \quad \forall x\in M, \varepsilon>0,
$$
where  $L(\varepsilon)$ is as in Lemma \ref{lemma:1}.
\end {lemma}

\begin{proof}
Set
$$
Q_{\mu}(f)=\left\{ x\in M :
    \lim\limits_{n\rightarrow \pm \infty}\frac{1}{n}
    \sum\limits_{k=0}^{n-1}\xi(f^kx)
   = \int _M \xi(x) \, d\mu(x), \quad \forall \xi\in C^0(M)\right\}.
$$
By the Birkhoff ergodic theorem, $Q_{\mu}(f)$ is an $f$-invariant and
 $\mu$-full measure set. For any $x\in M$ and
$\varepsilon>0$, take $L(\varepsilon)$ as in Lemma \ref{lemma:1}.
Let $y\in Q_{\mu}(f)$, then, by
the Birkhoff ergodic theorem and  Lemma \ref{lemma:1},
\begin{equation*}
 \begin{split}
\mu(B_M(x, \varepsilon))
    &=  \lim\limits_{n\rightarrow \infty}
           \frac{Card \{i\in \{0, 1, \ldots, n-1\}:
           f^i(y)\cap B_M(x,\varepsilon)
             \neq  \emptyset \} } {n} \\
    & \geq \frac{1}{L(\varepsilon)}>0.
\end{split}
\end{equation*}
\end{proof}

\subsection{Time-changed systems and entropy}

First, we recall the notion of additive function.

\begin{definition}\label{definition:2}
Suppose $\psi_t$ is a  measurable flow on  a Borel probability space
$(\Omega,\mathcal{B}, \mu)$ and  $\Omega$ is divided into disjoint
invariant measurable sets $A$ and $N$ such that $\mu (A)=1$ and $\mu
(N)=0$. Further suppose that $\theta(t, x)$ is a real measurable function defined on
$(-\infty, +\infty) \times (\Omega\setminus N) = \mathbb{R} \times A$
with the following properties for every fixed $x \in A$:
\begin{enumerate}
\item $\theta(t, x)$ is continuous and nondecreasing in t;
\item  $\theta(t+s, x)=\theta(s, x)+\theta(t, \psi_sx)$ for all $t$
and $s$;
\item  $\theta(0, x)=0, \lim\limits_{t\rightarrow \infty}\theta(t,
x)=\infty, \lim\limits_{t\rightarrow -\infty}\theta(t, x)=-\infty$.
\end{enumerate}
\noindent
Then $\theta$ is called an {\em additive function of $\psi_t$ with the
carrier $A$}. An additive function is said to be {\em integrable}, if it is integrable
in $\Omega$ for every fixed $t$.
\end{definition}

\begin {lemma}\label{lemma:4}
If $\psi_t$ is a  measurable flow on a Borel probability space
$(\Omega,\mathcal{B}, \mu)$ and $a(x)$ is a non-negative, integrable
function satisfying:
$$
E_{\mu}(a)= \int_{\Omega}a(x)\, d\mu(x)>0,
$$
then the function
$$
\theta(t, x) = \int_0^t a(\psi_sx) \, ds
$$
is  an integrable additive function.
\end{lemma}
For a proof see Theorem 3.1 in \cite{tot}.

\begin{definition}\label{remark:1}
The  function $\theta(t, x)$ in Lemma \ref{lemma:4} is called the
{\em additive function defined by  $a(x)$}.
\end{definition}

Let $\psi_t$ be a  measurable flow on a Borel probability space
$(\Omega,\mathcal{B}, \mu)$ and let $\theta$ be an additive
function of $\psi_t$. Define
$$
\varphi _tx=\psi_{\tau(t, x)}x,
 \quad \textrm{with} \quad \tau(t, x)=\sup\{s: \theta(s, x)\leq t\},
$$
for all $-\infty <t<\infty$ and all $x\in A$, where $A$ is the
carrier of $\theta$. We call $\widehat{\Omega}=\{x\in A: \theta(t,
x)>0$, $\forall t>0\}$ the {\em regular set} of $\theta$; in other
words, the set of non-singular points of $\varphi_t$. Let
$\mathcal{\widehat{B}}=\mathcal{B}\cap \widehat{\Omega}$. Then
$\varphi_t: (\widehat{\Omega}, \mathcal{\widehat{B}})\rightarrow
(\widehat{\Omega}, \mathcal{\widehat{B}})$ is a measurable flow
(see Lemma 4.1 in \cite{tot}), which is called the {\em time changed
system} of $\psi_t$ by $\theta$.

Let $\mu$ be a $\psi_t$ invariant probability measure and $a(x)$ be
a non-negative, integrable function. Define
$$
\widehat{\mu}(B)=\int _B \, d\widehat{\mu}(x) = \int_B a(x) \, d\mu(x)
$$
for all $B\in \mathcal{B}$. We get  an invariant measure
$\widehat{\mu}$ of the time changed system $\varphi_t$ (see
(\cite{mar}).

\begin{lemma}\label{lemma:6}(\cite{tot})
If $\psi_t$ is ergodic then any time changed flow $\varphi_t$ of
$\psi_t$ is also ergodic.
\end {lemma}

The following is Theorem 10.1  of \cite{tot}.

\begin{theorem}\label{theorem:2}
Let $\psi_t$ be an arbitrary measurable flow on a Borel probability
space $(\Omega, \mathcal{B}, \mu)$, and let $\theta$ be any
integrable additive function of $\psi_t$. If the flow $\varphi_t$ on
$(\widehat{\Omega}, \widehat{\mathcal{B}}, \widehat{\mu})$ is the
time changed flow of $\phi_t$ by $\theta$, then we have the
inequality
$$
h_{\widehat{\mu}}(\varphi_t)\widehat{\mu}(\widehat{\Omega})
     \leq h_\mu(\psi_t)\mu(\Omega)
$$
for all fixed $t$, where $h_{\widehat{\mu}}(\varphi_t)$ and
$h_\mu(\psi_t)$ denote the measure-theoretic entropies of the
homeomorphisms $\varphi_t $ and $\psi_t$, respectively. The
equality holds when $\psi_t$ is ergodic.
\end{theorem}
This theorem will be used to establish that the entropy of $\widehat{\varphi}$
is non-zero, by exhibiting an ergodic $\mu$ with positive entropy and
establishing that $\hat{\mu}$ is finite.

\begin{definition}
An ergodic measure is {\em atomic} if it supported on a periodic orbit.
\end{definition}

\begin{lemma}\label{lemma:3}
Assume that $X$ is a $C^1$ vector field on $\Omega$ and $\alpha\in
C^1(\Omega, [0,1])$ satisfies  {\bf (H)} (see \S 1). Denote by $\varphi_t$ the flow
induced by the vector field $\alpha X$ on $\Omega$. For any $x\in M$,
define $\gamma (x)$ by:
$$
\left\{
\begin{array}{ll}
\varphi_{\gamma(x)}(x,0)=\psi_1(x,0)=(f(x), 0),
        & \ \ (x,0)\neq f^{-1}(p)\ and\ (x,0)\neq p\,;\\
\gamma(x)=+\infty,
        & \ \  (x,0)= f^{-1}(p)\ or\ (x,0)= p.
\end {array}
\right.
$$
Then for any non-atomic  ergodic measure $\bar{\mu}$ of
$\varphi_t$, there exists a non-atomic ergodic measure $\mu$ of
$f$ such that
\begin{equation}\label{eq:9}
E_{\bar{\mu}}(\xi)
  = \frac{1}{E_{\mu}(\gamma)}E_{\mu} \left(\int_{0}^{\gamma(x)}
                \xi(\varphi_t(x,0)) \,dt \right),
       \quad \forall \xi\in C^0(\Omega),
\end{equation}
where $E_{\bar{\mu}}(\xi) = \int_{\Omega}\xi \, d\bar{\mu}$
and $E_{\mu}(\gamma) = \int_M \gamma(y) \, d\mu (y)$.
\end {lemma}
The proof is similar to Lemma 4 in \cite{ohno} and
thus omitted.

\begin{corollary}\label{corollary:1}
If $E_{\mu}(\gamma)=+\infty$ for all non-atomic  ergodic measures
$\mu$ of $f$, then $\varphi_t$ has only atomic  invariant Borel
probability measures.
\end{corollary}
\begin{proof}
Otherwise,  $\varphi_t$ should have  an non-atomic
invariant Borel probability measure $\bar{\mu}$. By Lemma
\ref{lemma:3}, there exists an non-atomic ergodic measure
$\mu$ of $f$ satisfying equation (\ref{eq:9}).

On the other hand, since $\bar{\mu}$ is non-atomic, one can
choose an open set $V\subset \Omega\setminus\{p\}$ such that
$\overline{V}\subset \Omega\setminus\{p\}$ and
$\bar{\mu}(V)>0$.  Let $U$ be an open set satisfying
$\overline{V}\subset U\subset \overline{U}\subset
\Omega\setminus\{p\}$. Without loss of generality, we can assume
that there exist two open sets $B_1\subset B_2\subset M$ and two
open intervals $I_1\subset I_2\subset [-1,1]$ such that
$$V=\{\psi_t(x,0)
:\ x\in B_1, t\in I_1\}\ \ and\ \ U=\{\psi_t(x,0):\ x\in B_2, t\in
I_2\}.$$ For any $q\in \overline{U}$, denote  $\tau(q)=\max\{t\in
\mathbb{R}^+:\ \varphi_{[0,t]}(q)\subset \overline{U}\}$. By the
continuity of $\varphi$ and the compactness of $\overline{U}$,
there exists $\tau>0$ such that $\tau(q)\leq \tau$ for any $q\in
\overline{U}$.

We define a continuous function $\xi:\ \Omega \to [0,1]$ such that
$\xi|_{V}\equiv 1$ and $\xi|_{\Omega\setminus U}\equiv 0$. Clearly,
$E_{\bar{\mu}}(\xi)>0$.

For $x\in B_2$, if $(x,0)\neq p$, we denote  $\tau_1(x)=\min\{t:\
\varphi_t(x, 0)\in \overline{U}\}$. If $x\in B_2$ and $(x,0)=p$,
let $\tau_1(x)=0$. By the definition of $\xi$, we  have
\begin{equation}
\begin{split}
E_{\mu}\left(\int_{0}^{\gamma(x)} \xi(\varphi_t(x,0)) \,dt\right)
 =& \int _{B_2}\int_{0}^{\gamma(x)}\xi(\varphi_t(x,0))\,dt\, d\mu\\
\leq&\int_{B_2}\int_{0}^{\tau_1(x)+\tau}\xi(\varphi_t(x,0))\,dt\, d\mu\\
=&\int _{B_2}\int_{0}^{\tau_1(x)}\xi(\varphi_t(x,0))\,dt\, d\mu \\
 & +
\int _{B_2}\int_{\tau_1(x)}^{\tau_1(x)+\tau}\xi(\varphi_t(x,0))\,dt\, d\mu\\
=&\int
_{B_2}\int_{\tau_1(x)}^{\tau_1(x)+\tau}\xi(\varphi_t(x,0))\,dt\,
d\mu\\
\leq&\tau.
\end{split}
\end{equation}
Thus $\frac{1}{E_{\mu}(\gamma)}E_{\mu} \left(\int_{0}^{\gamma(x)}
\xi(\varphi_t(x,0)) \,dt \right)=0$. This implies that $\mu$ and
$\bar\mu$ do not satisfy  equation (\ref{eq:9}). So a
contradiction is  deduced.
\end{proof}

Lemma~\ref{lemma:2} will be used to show that $E_{\mu}(\gamma)=+\infty$
for $\varphi$ and then Corollary~\ref{corollary:1} will imply that 
$\varphi$ has only atomic invariant probability measures and thus 
has zero entropy.

\subsection{Measure theoretic entropies  of $f$ and its standard suspension $\psi_t$}

First we present a link between the invariant measures of $f$ and $\psi_t$.

\begin{lemma}\label{lemma:8}
For any invariant probability measure $\mu$ of $f$ on $M$, we
define
$$
\int_{\Omega} \xi \, d\bar{\mu}
    := \int_M \int_0^1 \xi(x,t) \, dt d\mu, \quad  \forall \xi\in C^0(\Omega),
$$
then $\bar{\mu}$ is an invariant measure of $\psi_t$ on $\Omega$.
Further, $\bar{\mu}$ is ergodic if $\mu$ is ergodic.
\end{lemma}
The proof is
elementary and omitted.

That topological entropies  of a homeomorphism  and its standard
suspension coincide is a well known fact, see e.g. \cite{bow}. The
same is true for measure theoretic entropy, although, to our best
knowledge, we could  not find the original proof in literature.
For convenience of the reader we present a proof (see Proposition
\ref{th:5} below) using Katok's definition \cite{kat} of metric entropy.

For given $x\in M,$ $n\in \mathbb{N}$ and $\varepsilon>0$, we set
$$
D(x, n, \varepsilon, f)=\{y\in M :
     d(f^ix, f^iy)<\varepsilon, i=0, 1, \ldots, n-1\}
$$
and call it an $(n, \varepsilon, f)$-ball.

\begin{definition}\label{definition:4}(\cite{kat})
Given an ergodic measure $\mu$ on  $f:M\to M$ and a real
$\delta\in (0, 1)$, let $R(\delta, n, \varepsilon, f)$ denote the
smallest number of $(n, \varepsilon, f)$-balls needed to cover a
set whose $\mu$-measure is greater than $1-\delta$. Then the {\em
measure theoretic entropy of $f$}, denoted by $h_{\mu}(f)$, is
defined by
$$
h_{\mu}(f):=\lim\limits_{\varepsilon\to 0}\limsup\limits_{n\to
\infty}\frac{1}{n} \ln R(\delta, n, \varepsilon, f).
$$
\end{definition}
The quantity  $h_\mu(f)$ in Definition \ref {definition:4} is
independent of the choice of $\delta $ (see \cite{kat}). In a
similar way we define measure theoretic entropy for   $\psi$ as
follows. For $q\in \Omega$, $t\in \mathbb{R}$ and $\varepsilon>0$,
we define a $(t, \varepsilon, \psi)$-ball
$$
D(q, t, \varepsilon, \psi) =
    \{w\in \Omega: d(\psi_sw, \psi_sq)<\varepsilon,
         0\leq s\leq t\}.
$$
Given an ergodic measure $\mu$ of $\psi$ and $\delta\in (0, 1)$,
let $R(\delta, t, \varepsilon, \psi)$ denote the smallest number
of $(t, \varepsilon, \psi)$-balls needed to cover a set whose
$\mu$-measure is greater than $1-\delta$. Then the measure
theoretic entropy of $\psi$, denoted by $h_{\mu}(\psi)$, is
defined by
$$
h_{\mu}(\psi):=\lim\limits_{\varepsilon\to 0} \limsup\limits_{t\to
\infty}\frac{1}{t}\ln R(\delta, t, \varepsilon, \psi).
$$
One shows easily that $h_\mu(\psi)$ defined in this way coincides
with $h_\mu(\psi_1)$, the measure theoretic entropy for the time
one map.

\begin{prop}\label{th:5}
Let $\mu$ be a ergodic measure of $f$ on $M$, then we have
$$
h_{\bar{\mu}}(\psi)=h_{\mu}(f),
$$
where  $\bar{\mu}$ is defined as in Lemma \ref{lemma:8}.
\end{prop}

\begin{proof}
For given $q=(x, t)\in \Omega$, $n\in
\mathbb{N}$ and $\varepsilon>0$, we set
$$
D(q, n, \varepsilon, \psi)=\{w\in \Omega :
            d(\psi_iw, \psi_iq)<\varepsilon,\ i=0, 1, \ldots, n-1\}
$$
and
$$
\widetilde{D}(q, n, \varepsilon, \psi)=\{(y, s):
          y\in D(x,n,\varepsilon,f),\ t-\varepsilon<s<t+\varepsilon\},
$$
where $D(x, n, \varepsilon, f)$ is  defined  before Definition
\ref{definition:4}. Then by Definition \ref{definition:1} it
follows that
$$
\widetilde{D}(q, n, \varepsilon, \psi)=\{w\in \Omega:
         d(\psi_\tau w,\psi_\tau q) < \varepsilon,\ 0\leq \tau\leq n\}.
$$
We call $\widetilde{D}(q, n, \varepsilon, \psi)$ an $(n,
\varepsilon, \psi)$-box with center $q$. Let
$\widetilde{R}(\delta, n, \varepsilon, \psi)$ denote the smallest
number of $(n, \varepsilon, \psi)$-boxes needed to cover a set
whose $\bar{\mu}$-measure is greater than $1-\delta$. Then,
\begin{equation}\label{eq:6}
h_{\bar{\mu}}(\psi) = \lim\limits_{\varepsilon\to
0}\limsup\limits_{n\to \infty}\frac{1}{n}\ln \widetilde{R}(\delta,
n, \varepsilon, \psi).
\end{equation}
We will prove that $h_{\bar{\mu}}(\psi)= h_{\mu}(f)$ in two steps.

\medskip
\noindent
{\bf Step 1.}
For a subset $A\subset\Omega$, we denote
$$
\pi A=\{x\in M: (x, t)\in A, \textrm{for some } 0\leq t<1 \}.
$$
We have
$$
\bar{\mu}(A) = \int_M \int_0^1 \chi_A(x, t) \, dt \, d\mu
         \leq \int_M \chi_{\pi A}(x) \, d\mu
                = \mu(\pi(A)).
$$
Let $\varepsilon>0$ and $0<\delta<1$. If the boxes 
$\widetilde{D}(q_1,n,\varepsilon, \psi)$, 
$\widetilde{D}(q_2, n, \varepsilon, \psi)$, \ldots, $\widetilde{D}(q_{\widetilde{R}(\delta,n,\varepsilon,\psi)},n,\varepsilon,\psi)$ cover a set $A\subset \Omega$ and
$\bar{\mu}(A)\geq 1-\delta$, then $D(x_1, n, \varepsilon, f)$,
$D(x_2, n, \varepsilon, f)$, \ldots, 
$D(x_{\widetilde{R}(\delta,n,\varepsilon, \psi)}, n, \varepsilon, f)$ 
cover $\pi A$ and
$\mu(\pi A)\geq 1-\delta$, here $q_i=(x_i, t_i)$ for $i=1,
2,\ldots, \widetilde{R}(\delta, n, \varepsilon, \psi)$. 
So, $\widetilde{R}(\delta, n, \varepsilon \psi)\geq
R(\delta,n,\varepsilon,f)$ and thus $h_{\bar{\mu}}(\psi) \geq
h_{\mu}(f)$ by Definition \ref {definition:4} and (\ref {eq:6}).

\medskip
\noindent {\bf Step 2.} Let $\varepsilon>0$ and $0<\delta<1$. If
$(n, \varepsilon, f)$-balls
$$
D(x_1, n, \varepsilon, f), D(x_2, n, \varepsilon, f), \ldots,
D(x_{R(\delta, n, \varepsilon, f)}, n, \varepsilon, f)
$$
cover $B\subset M$ with $\mu(B)\geq 1-\delta$, then boxes
$$
\begin{array}{lll}
\widetilde{D}((x_1, 0), n, \varepsilon, \psi),& \ldots,&
\widetilde{D}((x_1, k(\varepsilon)), n, \varepsilon, \psi),\\
\widetilde{D}((x_2, 0), n, \varepsilon, \psi),& \ldots,&
\widetilde{D}((x_2, k(\varepsilon)), n, \varepsilon, \psi),\\ \ldots,&&\\
\widetilde{D}((x_{R(\delta, n, \varepsilon, f)}, 0), n,
\varepsilon, \psi),& \ldots,& \widetilde{D}((x_{R(\delta, n,
\varepsilon, f)}, k(\varepsilon)), n, \varepsilon, \psi)
\end{array}
$$
cover a set $A\subset \Omega$ with $\bar{\mu}(A)\geq 1-\delta,$
where $k(\varepsilon)$ is the smallest integer larger
than $\frac{1}{\varepsilon}$. So,
$$
\widetilde{R}(\delta, n, \varepsilon, \psi)\leq
k(\varepsilon)R(\delta, n, \varepsilon, f).
$$
Then by Definition~\ref{definition:4} and (\ref{eq:6})  we
conclude that $h_{\bar{\mu}}(\psi)\leq h_{\mu}(f)$.
\end{proof}


\section{Proof of Main Theorem}

Recall that $M$ denotes a smooth compact Riemannian manifold  of
dimension at least 2. In the rest of the paper we suppose that $f:
M\to M$ is  a $C^\infty$  minimal diffeomorphism, it preserves a
measure $\mu$ and  has positive measure-theoretic entropy
$h_{\mu}(f)>0$ (and thus has a positive topological entropy). See
\cite{her} for the existence of such a diffeomorphism. By the
variational principle we can assume that $\mu$ is an ergodic
measure.

As in Definition~\ref{definition:1} let $\Omega$ be the suspension
manifold with the standard differentiable suspension flow $\psi:
\Omega\times \mathbb{R}\to \Omega$. Denote by $X$ the $C^\infty$
vector field on $\Omega$ which induces $\psi$. We will construct a
function ${\alpha}\in C^\infty(\Omega,[0,1])$ and a $C^\infty$
vector field $Y=\alpha X$ on $\Omega$ with  zero topological
entropy in Theorem~\ref{theorem:3}. In Theorem~\ref{theorem:4}
we construct a function $\widehat{\alpha}\in
C^\infty(\Omega,[0,1])$ and corresponding $C^\infty$ vector field
$\widehat{Y}=\widehat{\alpha} X$ on $\Omega$ with positive
topological entropy. Both $\alpha$ and $\widehat{\alpha}$ will be
constructed satisfying condition {\bf(H)} and by
Proposition~\ref{remark:2} the two differentiable flows induced by
$Y$ and $\widehat{Y}$ are equivalent; thus completing the Main
Theorem.

Before we prove Theorem~\ref{theorem:3} below, we will need the following.
\begin{lemma}\label{lemma:7}
For a given a sequence of positive numbers
$1=\beta_{-1}>\beta_0>\beta_1>\cdots>\beta_i>\cdots$, there exists
a $C^\infty$ function $\omega: B^n(2)\rightarrow [0,1]$
such that
\begin{enumerate}
\item $\omega(x)=0$ if and only if $x=0$;
\item $\| \omega|_{B^n(\frac{1}{i+1})} \|_{C^0}
         \leq \beta_{i-1}$, $i=0, 1, 2, \ldots$ ;
\item $\omega|_{B^n(2) \setminus B^n(1)}=1$.
\end{enumerate}
\end{lemma}

\begin{proof} Without loss of  generality, we assume that
$\lim_{i\rightarrow \infty} \beta_i = 0$.

Let $h(t)$ be the function:
$$
   h(t) = \left\{\begin{array}{ll}
                e^{-1/t}, \quad & 0< t \le 1, \\
                0              , \quad & -1 < t \le 0.
                \end{array}
          \right.
$$
Clearly $h(t)$ is $C^\infty$ smooth and has zero
derivatives of all orders at $0$.
Let $\{\beta_i\}$ be as above and suppose $\{\alpha_i\}$
is any decreasing sequence of positive numbers
$1=\alpha_{-1}>\alpha_0>\alpha_1>\cdots>\alpha_i>\cdots$,
with $\lim_{i \rightarrow \infty} \alpha_i = 0$.
For $t < \alpha_0$, let $\eta(t)$ to be the function on
$[-1,\alpha_0]$ defined by the series:
$$
   \eta(t) = \sum_{i=1}^\infty 2^{-i-1} \beta_{i-1} h(t-\alpha_i).
$$
This series is monotone increasing in $i$ and converges uniformly.
It is zero on $[-1,0]$ and  positive  on $(0,\alpha_0]$.
For any $k$ and $0 < t < \alpha_k$ we have
\begin{equation}
    \eta(t) = \sum_{i=k+1}^\infty 2^{-i-1}  \beta_{i-1}  h(t-\alpha_i)
             < \frac{\beta_k h(1)}{2^k}.
\end{equation}
Further, since the derivative of the partial sums of this series converge
uniformly, we may take the derivative of the sum and we have that:
$$
     \eta'(t) = \sum_{i=1}^\infty 2^{-i-1} \beta_{i-1} h'(t-\alpha_i).
$$
This series also is monotone increasing and converges uniformly.
For $0 < t < \alpha_k$ we obtain
\begin{equation}
    \eta'(t)  = \sum_{i=j+1}^\infty 2^{-i-1}  \beta_{i-1}  h'(t-\alpha_i)
       < \frac{\beta_k h'(1)}{2^k}.
\end{equation}
Thus, $\lim_{t\rightarrow 0^+} \eta'(t) = 0$.
By induction, we may also conclude that
$$
     \eta^{(k)}(t) = \sum_{i=1}^\infty 2^{-i-1} \beta_{i-1} h^{(k)}(t-\alpha_i)
$$
converges uniformly and
\begin{equation}\label{limder}
\lim_{t\rightarrow 0^+} \eta^{(k)}(t) = 0.
\end{equation}
We can now clearly extend $\eta(t)$ to the interval $[0,2]$ in
such a way that $\eta$ is $C^\infty$ and $\eta(t) = 1$ for $t \in [1,2]$.

To finish the proof of the Lemma, we set $\alpha_i = (i+1)^{-1}$ in the
above construction of $\eta(t)$
and let $\omega(x) = \eta(|x|)$ on $B^n(2)$. Because of the construction and
equation (\ref{limder}) this function is $C^\infty$ smooth at $0$ and on $B^n(2)$.
\end{proof}

\begin {theorem}\label{theorem:3}
There exists a function ${\alpha}\in
C^\infty(\Omega,[0,1])$, satisfying {\bf (H)}, such that the flow
defined by the vector field ${Y}={\alpha} X$
has zero topological entropy.
\end{theorem}

\begin{proof}
Recall that $p=[(x_0, 0)]=\pi(x_0, 0)$ in condition {\bf(H)} where  $\pi$
is the quotient map $\pi: M\times \mathbb{R}\rightarrow \Omega$ for
the suspension.

Without loss of generality, we can assume the  existence of   a
coordinate chart $(\widetilde{V}, \xi)$  of $\Omega$ satisfying the
following.

(i)\ There exists an open set $V$ of $\Omega$, such that $p\in V$
and $\bar{V}\subset \widetilde{V}$;

(ii)\ $\xi(p)=0$, $\xi(V)=B^{m+1}(1),$
$\xi(\widetilde{V})=B^{m+1}(2)$, where $2\leq m=dim M$;

(iii)\ $\exists$ $i_1\in \mathbb{N}$ such that
$$
   \textrm{Closure}\left( \pi(B_M(x_0,i_1^{-1})\times\{0\}) \right) \subset V
$$
and
$$
\xi(\pi(B_M(x_0,i_1^{-1})\times\{0\}))\subset \mathcal{R}=\{x=(x_1,
\ldots, x_m,  x_{m+1}): x_{m+1}=0\}.
$$

(iv)\ $\exists$ $i_2\in \mathbb{N}$ such that
$$
B_{\Omega}(p,i_2^{-1})\subset V
$$
and
$$
\xi(B_{\Omega}(p,i^{-1}))=B^{m+1}(i^{-1})
$$
for any $i_2<i\in \mathbb{N}$.

Under these assumptions,  there exists  $i_3\in \mathbb{N}$ and
$i_2<i_3$, with the property that for any $i\geq 0$ there exists
$1\gg l_{i_3+i}>0$ such that
$$
\textrm{Closure} \left( \pi (B_M(x_0,
\frac{1}{i_3+i})\times[-l_{i_3+i}, 0]) \right)
            \subset B_{\Omega}(p, \frac{1}{i_2+i}).
$$
We set  $i_0:=\max \{i_1, i_2, i_3\}$. For any $i>i_0$, by 
Lemma~\ref{lemma:2}, there exists
$L(\frac{1}{i})$ such that for any ergodic measure $\tau$ of $f$, we
have
$$
\tau(B_M(f^{-1}x_0, \frac{1}{i}))\geq
\frac{1}{L(\frac{1}{i})}:=\delta(i)>0.
$$ 
We define $\beta_{-1}:=1$ and
$\beta_{i-1}:=\frac{l_{i_0+i}}{i_0+i}\delta(i_0+i)$ for $i\geq 1$.

Using Lemma \ref{lemma:7}, one can find a $C^\infty$ function
$\omega: \xi(\widetilde{V})\rightarrow [0,1]$ with the properties:

\smallskip
(i)\ $\omega|_{\xi(\widetilde{V}\setminus V)}\equiv 1$;

\smallskip
(ii)\ $\parallel
\omega|_{B^{m+1}(\frac{1}{i_2+i})}\parallel_{C^0}\leq \beta_{i-1}$;

\smallskip
(iii)\ $\omega(0)=0$ and $0<\omega(a)\leq1$ for $0\neq a\in
\xi(\widetilde{V})$.

\smallskip
\noindent We then define a function ${\alpha} $ in  $
C^\infty(\Omega, [0,1])$ as follows
$$
{\alpha}(q):= \left\{
\begin{array}{ll}
\omega\circ \xi(q),&\ \ q\in \widetilde{V};\\
 1, &\ \ q\in \Omega\setminus \widetilde{V}.
\end {array}
\right.
$$
Then
$$
\left\| {\alpha}|_{B_{\Omega}(p, \frac{1}{i_2+i})} \right\|_{C^0}
  \, = \sup\limits_{x\in B_{\Omega}(p, \frac{1}{i_2+i})}\{{\alpha}(x)\}
       \leq \beta_{i-1},
$$
where we assume, without loss of generality, that $\|\xi\|_{C^0}\leq
1.$ We then define ${Y}:={\alpha}X$ and let ${\varphi}_t$ denote the
flow induced by ${Y}$.

Recall the function  $\gamma: M\to \mathbb{R}\cup \{\infty\} $ in
Lemma \ref{lemma:3} and observe that for any $x\in B_M(f^{-1}(x_0),
\frac{1}{i_2+i})$,
$$
l_{i_0+i} = \int^{\gamma(x)}_{t(x)}
     \sqrt{<{\alpha}({\varphi}_s(x))X({\varphi}_s(x)),
            {\alpha}({\varphi}_s(x))X({\varphi}_s(x))>} \, ds,
$$
where $t(x)>0$ satisfies ${\varphi}_{t(x)}(x) =
\psi_{1-l_{i_0+i}}(x)$. Then
$$
\gamma(x)\geq \gamma(x)-t(x)\geq\frac {l_{i_0+i}}{\|
{\alpha}|_{B_{\Omega}(p, \frac{1}{i_2+i})} \| \| X\|}\geq \frac
{l_{i_0+i}}{\beta_{i-1} \|X\|} = \frac{i_0+i}{\delta(i_0+i)\|X\|  }
$$ for any $x\in B_M(f^{-1}(x_0), \frac{1}{i_0+i})$. Therefore,
$$
\gamma|_{B_M(f^{-1}(x_0), \frac{1}{i_0+i})}\geq
\frac{i_0+i}{\delta(i_0+i)\|X\|}.
$$

Now for any ergodic measure  $\tau$  of $f$,
$$
E_{\tau}(\gamma)=\int_M \gamma(x) \, d\mu(x)
          \geq \frac{i_0+i}{\delta(i_0+i)\|X\|}
                \tau( B_M(f^{-1}(x_0), \frac{1}{i_0+i}))
           \geq \frac{i_0+i}{\|X\|}\to +\infty
$$
as $i\to +\infty$. By Corollary \ref{corollary:1}, all ergodic
measures of the flow ${\varphi}$ induced by ${Y}={\alpha} X$ are
atomic. The topological entropy $h(\varphi)$ equals $h(\varphi_1)$
by definition, while $h(\varphi_1)$ coincides with the supremum of
measure theoretic entropies $h_m(\varphi_1)$ over all $\varphi$
invariant and ergodic measures $m$ by Theorem A of \cite{sun} (note
that a $\varphi$-invariant and ergodic measure is
$\varphi_1$-invariant but not necessarily $\varphi_1$-ergodic). So
the topological entropy of $\varphi$ is zero.
\end{proof}

\begin{theorem}\label{theorem:4}
There exists a function $\widehat{\alpha}\in C^\infty(\Omega,[0,1])$,
satisfying {\bf (H)}, such that $\widehat{Y}=\widehat{\alpha} X$
has positive topological entropy.
\end {theorem}

\begin{proof}
Take $p\in \Omega$ as in Theorem \ref{theorem:3}. Recall that
$f: M\to M$ is
a minimal diffeomorphism  which preserves   an
ergodic  measure $\mu$ with $h_{\mu}(f)>0.$ Let $\bar \mu$
denote the ergodic measure on $(\Omega, \psi)$ defined in Lemma
\ref{lemma:8}. By Theorem~\ref{th:5},  $h_{\bar\mu}(\psi)=h_\mu(f)>0$.
Noting that $\dim M\geq 2$ and thus that $\dim \Omega \geq 3$, we can
find a function $\widehat{\alpha} \in C^\infty(\Omega, [0,1])$
that satisfies not only  the assumption  {\bf (H)} stated in Section 2  but
also the following:
\begin{equation}\label{eq:8}
\int_{\Omega} \frac{1}{\widehat{\alpha}(x)} \, d\bar{\mu}(x) = K < \infty.
\end{equation}
In fact, we can assume that the coordinate $(V, \xi)$
satisfies $\xi(V)=B^{m+1}(2)$ and
$\xi(V')=B^{m+1}(1)$, here $V'$ is an open set of
$\Omega$ satisfying $p\in V'$ and $\overline{V'}\subset V$. We
select a smooth function $\omega: \mathbb{R}^{m+1}\to \mathbb{R}$
satisfying
\begin{enumerate}
\item $\omega(x)= \|x\|^2 := (x_1)^2+(x_2)^2+\cdots +(x_{m+1})^2$, \ if
 $\|x\| \leq
\frac{1}{2}$;
\item $\frac{1}{4}<|\omega(x)|<2$, \ if $\frac{1}{2}<||x||<1$;
\item $\omega(x)=1$, \ if $||x||\geq 1$.\\
\end{enumerate}
It is easy to see that $\frac{1}{\omega(x)}$ is  integrable on
$B^{m+1}(2)$ since $m+1\geq 3$. Then the function
$$\widehat{\alpha}(q)=
\left\{
\begin{array}{cl}
\omega\circ \xi^{-1}&\ \ p\in V;\\
1&\ \ p\in M\setminus V
\end{array}
\right.
$$
satisfies assumption {\bf (H)} and (\ref{eq:8}).
By Lemma \ref{lemma:4}, the function
$$
 \theta(t,x) = \int_0^t \frac{1}{\widehat{\alpha}(\psi_sx)} \, ds
$$
is  an integrable additive function. Consider  the regular set
$\widehat{\Omega}$ of $\theta$ and construct the corresponding
Borel $\sigma$ algebra $\widehat{B}$ (see the  description  after
Remark \ref {remark:1}). Denote by $\widehat{\varphi}_t$ the time
changed system of $\psi_t$.

One can easily verify that $\widehat{\varphi}_t$ is the flow
induced by the vector field $\widehat{Y}:=\widehat{\alpha} X$. As
discussed in Section 2, we can define a measure on
$(\widehat{\Omega}, \widehat{\mathcal{B}})$:
$$
\widehat{\mu}(B) = \int _B \, d\widehat{\mu}(x)
= \int_B \frac{1}{\widehat{\alpha}(x)} \, d\bar{\mu}(x)
$$
for all $B\in \widehat{\mathcal{B}}$. By Lemma~\ref{lemma:6},
$\widehat{\mu}$ is an invariant ergodic measure of $\widehat{\varphi}_t$ and
$$
 \widehat{\mu}(\widehat{\Omega})
    \leq \int_{\Omega} \frac{1}{\widehat{\alpha}(x)} \, d\bar{\mu}(x) = K < \infty.
$$
Since $\bar \mu$ is  ergodic and by Theorem \ref{theorem:2}, we have
$$
h_{\widehat{\mu}}(\widehat{\varphi_1})\widehat{\mu}(\widehat{\Omega})
    = h_{\bar\mu}(\psi_1)\bar{\mu}(\Omega).
$$
So,
$h_{\widehat{\mu}}(\widehat{\varphi})=h_{\widehat{\mu}}(\widehat{\varphi}_1)
>0$.
\end{proof}

\medskip

\noindent
{\bf Proof of Theorem \ref{theorem:1}.} \\
The flow induced by the $C^\infty$
vector field $Y=\alpha X$  has zero entropy and the flow
induced by the $C^\infty$ vector field $\widehat{Y}=\widehat{\alpha} X$ has
positive entropy. They are equivalent by Proposition~\ref{remark:2}.
\hfill $\Box$

\medskip

\noindent
{\bf Final note.} \\
The map $f:M \rightarrow M$ provided by \cite{her} is real analytic.
However, our method of proof clearly cannot be made analytic
since $\alpha$ is flat at $p$. It seems likely that zero is an invariant
of equivalent analytic flows.

\appendix
\section*{Acknowledgement}
The authors are grateful to the referee for improving the manuscript 
and to David Epstein for suggesting the idea of Lemma 3.1. The second 
author thanks the Mathematics Research Centre at Warwick University for 
their hospitality during the period the work was completed.

\bibliographystyle{amsplain}

\end{document}